\documentclass[10pt,a4paper]{amsart}
\usepackage{a4}
\usepackage{amsfonts, amssymb}
\usepackage[all]{xy}
\usepackage{tikz}
\usetikzlibrary{shapes.geometric,positioning}
\def\a{{\mathbf a}}

\def\b{{\mathbf b}}

\newtheorem{theorem}{Theorem}

\newtheorem{remark}{Remark}
\newtheorem{lemma}{Lemma}

\numberwithin{cor}{section}

\begin{document}
\title{Gaps in  sumsets of $s$ pseudo $s$-th power sequences}
\begin{abstract}We study the length of the gaps between consecutive members in the sumset $sA$ when $A$ is a pseudo $s$-th power sequence, with $s \ge 2$. We show that, almost surely, $\lim\sup (b_{n+1}-b_{n})/\log (b_n) = s^s s!/\Gamma^s(1/s)$, where $b_n$ are the elements of $sA$.
\end{abstract}
\author{Javier Cilleruelo}

\thanks{The first author was supported by grants MTM 2011-22851 of MICINN  and ICMAT Severo Ochoa project SEV-2011-0087.\\  \indent Both authors are thankful to Ecole Polytechnique which made their collaboration easier.}
\address{Instituto de Ciencias Matem\'aticas (CSIC-UAM-UC3M-UCM) and
Departamento de Matem\'aticas\\
Universidad Aut\'onoma de Madrid\\
28049, Madrid, Espa\~na} \email{franciscojavier.cilleruelo@uam.es}
\author{ Jean-Marc Deshouillers }
\address{Bordeaux INP\\
Institut Math\'{e}matique de Bordeaux\\
33405 Talence, France}
 \email{jean-marc.deshouillers@math.u-bordeaux.fr}
\subjclass{2000 Mathematics Subject Classification: 11B83.}
\keywords{Additive Number Theory, Pseudo $s$-th powers, Probabilistic method}
\maketitle

\section{Introduction}

Erd\H os and R\'{e}nyi \cite{ER60} proposed in 1960 a probabilistic model for sequences $A$ growing like the $s$-th powers: they build a probability space $(\mathcal{U}, \mathcal{T}, P)$ and a sequence of independent random variables $(\xi_n)_{n \in \mathbb{N}}$ with values in $\{0, 1\}$ and $P(\xi_n =1)= \frac 1s n^{-1+1/s}$; to any $u \in \mathcal{U}$, they associate the sequence of integers $A = A_u$ such that $n \in A_u$ if and only if $\xi_n(u) = 1$. In short, the events $\{n \in A\}$ are independent and $P(n \in A) = \frac 1s n^{-1+1/s}$. The counting function of these random sequences $A$  satisfies almost surely the asymptotic relation $|A\cap [1, x]|\sim x^{1/s}$, whence the terminology \textit{pseudo $s$-th powers}.
 Erd\H os and R\'{e}nyi studied the random variable $r_s(A,n)$ which counts the number of representations of $n$ in the form $n=a_1+\cdots +a_s,\ a_1\le \cdots \le a_s,\ a_i\in A$. For the simplest case $s=2$ they proved that $r_2(A,n)$ converges to a Poisson distribution with parameter $\pi/8$, when $n\to\infty$. They also claimed the analogous result for $s>2$ but their analysis did not take into account the dependence of some events. J. H. Goguel \cite{G75} proved indeed that for each integer $d$, the sequence of the integers $n$ such that $r_s(A,n)=d$ has almost surely the density $\lambda_s ^d e^{-\lambda_s} / d!$, where $\lambda_s =\Gamma^s(1/s)/(s^ss!)$. B. Landreau \cite{L95} gave a proof of this result based on correlation inequalities and also showed that the sequence of random variables $(r_s(A,n))_n$  converges in law towards the Poisson distribution with parameter $\lambda_s$.\\
\indent In particular,  both the sets of the integers belonging, or not belonging, to
$sA=\{a_1 + \cdots + a_s \; : \; a_i \in A\}$
have almost surely a positive density and it makes sense to study the length of the gaps in $sA$. The aim of the paper is to obtain a precise estimate for the maximal length of such gaps.
\begin{theorem}\label{main}
For any $s\ge 2$ the  sequence $sA=(b_n)_n$, sum of $s$ copies of a pseudo $s$-th power sequence $A$, satisfies almost surely
\begin{equation}\label{limit}\limsup_{n\to \infty}\frac{b_{n+1}-b_n}{\log b_n}= \frac{s^ss!}{\Gamma^s(1/s)}.
\end{equation}
\end{theorem}

We remark that this result is heuristically consistent with the easier fact that for a random sequence $S$ with $P(n \in S) = 1 - e^{-\lambda}$, we have $\lim\sup (s_{m+1}-s_{m})/\log s_m = 1/\lambda$ almost surely.

\section{Notation and general lemmas}

\subsection{Notation}
\bigskip
We retain the notation of the introduction, for the probability space $(\mathcal{U}, \mathcal{T}, P)$ and the definition of the random sequences $A = A_u$, where the events $\{n \in A\}$ are independent and $P(n \in A) = \frac 1s n^{-1+1/s}$.
We further use the following notation.
\begin{itemize}
\item [i)] We write $\omega$ to denote a  set of distinct integers and we denote by $E_{\omega}$ and $E_{\omega}^c$ the events
$$E_{\omega}:=\{\omega \subset A\} \qquad \text{ and } \qquad E_{\omega}^c:=\{\omega \not \subset A\} $$
respectively. We write $\omega \sim \omega'$ to mean that $\omega \cap \omega'\ne \emptyset$ but $\omega \ne \omega'$: e remark that $\omega\sim\omega'$ if and only if the events $E_{\omega}$ and $E_{\omega'}$ are distinct an dependent.

If $\omega=\{x_1,\dots,x_r\}$ we write
$$\sigma(\omega)=\{a_1x_1+\cdots +a_rx_r:\  a_1+\cdots +a_r=s,\ a_i\ge 1\}$$
for the set of all integers that can be written as a sum of $s$ integers using all the integers $x_1,\dots,x_r$. We denote by $\Omega_z$ the family of sets
$$\Omega_z=\{\omega:\ z\in \sigma(\omega)\}.$$
  \item [ii)]Given $\alpha>0$, we denote by $I_i$ the interval $[i,i+\alpha\log i]$ and
  we denote by $F_i$ the event
  $$F_i:= \{sA\cap I_i=\emptyset\}.$$
  We denote by $\Omega_{I_i}$ the family of sets $$\Omega_{I_i}=\{\omega:\ \sigma(\omega)\cap I_i\ne \emptyset\}.$$

  \item [iii)]We finally let $\lambda_s=\frac{\Gamma^s(1/s)}{s!s^s}$.
\end{itemize}

\subsection{Probabilistic lemmas}
\bigskip

We use the following generalization of the Borel-Cantelli Lemma, proved indeed by P. Erd\H{o}s and A. R\'{e}nyi in 1959 \cite{ER59}.

\begin{theorem}[Borel-Cantelli Lemma]\label{Borel}Let $(F_i)_{i \in \mathbb{N}}$ be a sequence of events and let $Z_n=\sum_{i\le n}P(F_i)$.\\

If the sequence $(Z_n)_n$ is bounded, then, with probability $1$, only finitely many of the events $F_i$ occur.\\

If the sequence $(Z_n)_n$ tends to infinity and
\begin{equation}\notag
\lim_{n\to \infty}\frac{\sum_{1\le i<j\le n} P(F_i\cap F_j)-P (F_i)P (F_j)}{Z_n^2}=0,
\end{equation}
then, with probability $1$, infinitely many of the events $F_i$ occur.
\end{theorem}

\begin{theorem}[Janson's Correlation Inequality \cite{BS89}]\label{des1}
Let $(E_{\omega})_{\omega\in \Omega}$ be a finite collection of events which are intersections of elementary independent events and assume that $P(E_{\omega})\le 1/2$ for any $\omega\in \Omega$.  Then
$$\prod_{\omega\in \Omega}P(E_{\omega}^c)\le P\big (\bigcap_{\omega\in \Omega}E_{\omega}^c\big )\le \prod_{\omega\in \Omega}P(E_{\omega}^c)\times \exp \big (2\sum_{\omega \sim \omega'}P(E_{\omega} \cap E_{\omega'})\big ),$$
where $\omega \sim \omega'$  means that the events $E_{\omega}$ and $E_{\omega'}$ are distinct and dependent.
\end{theorem}

\subsection{A technical lemma}
\begin{lemma}\label{t1} Given $1\le t\le s-1$ and
positive integers $a_1,\dots,a_t$ we have, as $z$ tends to infinity:
\begin{itemize}
  \item [i)]$$\sum_{\substack{x_1,\dots,x_t\\ a_1x_1+\cdots +a_tx_t=z}}(x_1\cdots x_t)^{-1+1/s}\ll z^{-1+t/s}.$$
  \item [ii)]$$\sum_{\substack{x_1,\dots,x_t\\ a_1x_1+\cdots +a_tx_t<z}}(x_1\cdots x_t)^{-1+1/s}\big (z-(a_1x_1+\cdots +a_tx_t)\big )^{-2t/s}\ll z^{-1/s}\log z.$$
  \item [iii)]$$\sum_{\substack{1\le x_1<\cdots <x_s\\x_1+\cdots +x_s=z}}(x_1\cdots x_s)^{-1+1/s}\sim s^s\lambda_s.$$
\end{itemize}
\end{lemma}
\begin{proof}
i) We have \begin{eqnarray*}\sum_{\substack{x_1,\dots,x_t\\ a_1x_1+\cdots +a_tx_t=z}}(x_1\cdots x_t)^{-1+1/s}
&= &(a_1\cdots a_t)^{1-1/s}\sum_{\substack{x_1,\dots,x_t\\ a_1x_1+\cdots +a_tx_t=z}}(a_1x_1\cdots a_tx_t)^{-1+1/s}\\
&\le &(a_1\cdots a_t)^{1-1/s}\sum_{\substack{y_1,\dots ,y_t\\ y_1+\cdots +y_t=z}}(y_1\cdots y_t)^{-1+1/s}.\end{eqnarray*}
If $y_1+\cdots +y_t=z$ then at least one of them, say $y_t$, is greater than $z/t$ and is determined by $y_1,\dots,y_{t-1}$. Thus,
\begin{eqnarray*}\sum_{\substack{x_1,\dots,x_t\\ a_1x_1+\cdots +a_tx_t=z}}(x_1\cdots x_t)^{-1+1/s}&\ll &z^{-1+1/s}\sum_{y_1,\dots ,y_{t-1}<z}(y_1\cdots y_{t-1})^{-1+1/s}\\ &\ll & z^{-1+1/s}\left (\sum_{y<z}y^{-1+1/s}\right )^{t-1}\\
&\ll & z^{-1+1/s}(z^{1/s} )^{t-1}\ll z^{-1+t/s}.\end{eqnarray*}

ii) We have \begin{eqnarray*}& &\sum_{\substack{x_1,\dots,x_t\\ a_1x_1+\cdots +a_tx_t<z}}(x_1\cdots x_t)^{-1+1/s}\big (z-(a_1x_1+\cdots +a_tx_t)\big )^{-2t/s}\\ &=&\sum_{m<z}(z-m)^{-2t/s}\sum_{\substack{x_1,\dots,x_t\\ a_1x_1+\cdots +a_tx_t=m}}(x_1\cdots x_t)^{-1+1/s}
\\ (\text{ by i) })& \ll &\sum_{m<z}(z-m)^{-2t/s}m^{-1+t/s}\\
&\ll &\sum_{m\le z/2}(z-m)^{-2t/s}m^{-1+t/s}+\sum_{z/2<m<z}(z-m)^{-2t/s}m^{-1+t/s}\\
&\ll & z^{-2t/s}z^{t/s}+z^{-1+t/s}\sum_{z/2<m<z}(z-m)^{-2t/s}\\
&\ll &z^{-t/s}+z^{-1+t/s}\left (1+\log z+z^{1-2t/s}\right )\\
&\ll &z^{-t/s}+z^{-1+t/s}\log z\\
&\ll &z^{-1/s}\log z. \end{eqnarray*}
\begin{remark} Except in the case when $s=2$ and $t=1$, the upper bound in ii) may be replaced by $z^{-1/s}$.
\end{remark}
iii) It follows from Lemma 3 of \cite{L95}.

\end{proof}

\section{Proof of Theorem \ref{main}}
\subsection{Combinatorial lemmas}
\begin{lemma}\label{d} We have
$$\sum_{\omega\in \Omega_z}P(E_{\omega})\sim \lambda_s$$ as $z\to \infty$.
\end{lemma}
\begin{proof}
\begin{equation}\label{sum}\sum_{\omega\in \Omega_z}P(E_{\omega})=\sum_{\substack{\omega \in \Omega_z\\ |\omega|=s}}P(E_{\omega})+\sum_{\substack{\omega \in \Omega_z\\ |\omega|\le s-1}}P(E_{\omega}).\end{equation}
The main contribution comes from the first sum.
$$\sum_{\substack{\omega \in \Omega_z\\ |\omega|=s}}P(E_{\omega})=\frac 1{s^s}\sum_{\substack{1\le x_1<\dots <x_s\\ x_1+\cdots +x_s=z}}(x_1\cdots x_s)^{-1+1/s}\sim \lambda_s$$ as $z\to\infty$, by Lemma \ref{t1} iii). For the second sum we have
\begin{eqnarray*}\sum_{\substack{\omega \in \Omega_z\\ |\omega|\le s-1}}P(E_{\omega})&\le &\sum_{r\le s-1}\sum_{\substack{a_1,\dots, a_r\\a_1+\cdots +a_r=s}}\sum_{a_1x_1+\cdots+a_rx_r=z}(x_1\cdots x_r)^{-1+1/s}\\
(\text{Lemma }\ref{t1}, i))\ &\ll & \sum_{r\le s-1}z^{\frac rs-1}\ll z^{-1/s}.\end{eqnarray*}.
\end{proof}
\begin{lemma}\label{dis}
For any $z\le z'$ we have
$$\sum_{\substack{\omega\sim \omega'\\ \omega\in \Omega_z,\ \omega'\in \Omega_{z'}}}P(E_{\omega}\cap E_{\omega'})\ll z^{-1/s}\log z.$$
\end{lemma}
\begin{proof}  If $\omega\in \Omega_z$ then  there exist some $r\le s$ and some positive integers $a_1,\dots,a_r$ with $a_1+\cdots +a_r=s$ such that $a_1x_1+\cdots +a_rx_r=z$.
Thus,  any  pair of sets $\omega\sim \omega'$ with $\omega\in \Omega_z,\ \omega'\in \Omega_{z'},\ z\le z'$ is of the form
\begin{eqnarray*}
\omega=\{x_1,\dots,x_t,u_{t+1},\dots, u_{r}\}\\
\omega'=\{x_1,\dots,x_t,v_{t+1},\dots, v_{r'}\}
\end{eqnarray*}
with $1\le t\le r, r'\le s$ and  positive integers $a_1,\dots,a_r$ and $b_1,\dots,b_{r'}$ with
\begin{eqnarray*}
a_1x_1+\cdots +a_tx_t+a_{t+1}u_{t+1}+\cdots +a_r u_r=z\ \\
b_1x_1+\cdots +b_tx_t+b_{t+1}v_{t+1}+\cdots +b_{r'} v_{r'}=z'.
\end{eqnarray*}
Of course if $r=t$ then $\omega=\{x_1,\dots,x_r\}$ and $r'\ge t+1$. Otherwise $\omega=\omega'$. And similarly, when $r'=t$, we have $r\ge t+1$.

Given $z,z',t,r,r',a_1,\dots,a_r,b_1,\dots,b_{r'}$ we estimate the sum $$\sum_{\substack{\omega\sim \omega'}}^*P(E_{\omega} \cap E_{\omega'})$$
where the sum is extended to the pairs $\omega\sim \omega'$ satisfying  the above conditions. We distinguish several cases according to the values of $r$ and $r'$.

\begin{itemize}
  \item If $r\ge t+1$ and $r'\ge t+1$, we have
\begin{eqnarray*}\sum_{\substack{\omega\sim \omega'}}^*P(E_{\omega} \cap E_{\omega'})& &\\
 \le  \sum_{\substack{x_1,\dots,x_t\\ a_1x_1+\cdots +a_tx_t<z\\b_1x_1+\cdots +b_tx_t<z' } }(x_1\cdots x_t)^{-1+1/s}&\times &\left (\sum_{\substack{u_{t+1},\dots, u_r,\\a_{t+1}u_{t+1}+\dots +a_ru_r\\ = z-(a_1x_1+\cdots +a_tx_t  )    }}(u_{t+1}\cdots u_r)^{-1+1/s}\right )\\
& \times &\left (\sum_{\substack{v_{t+1},\dots, v_{r'},\\b_{t+1}v_{t+1}+\dots +b_{r'}v_{r'}\\ =z'-(b_1x_1+\cdots +b_tx_t  )   }}(v_{t+1}\cdots v_{r'})^{-1+1/s}\right )\\
\end{eqnarray*}
By Lemma \ref{t1} i) we have
\begin{eqnarray*}& &\sum_{\substack{\omega\sim \omega'}}^*P(E_{\omega} \cap E_{\omega'})\nonumber\\
& \ll & \sum_{\substack{x_1,\dots,x_t\\ a_1x_1+\cdots +a_tx_t<z\\b_1x_1+\cdots +b_tx_t<z' } }(x_1\cdots x_t)^{-1+\frac 1s}\left (z-(a_1x_1+\cdots +a_tx_t)  \right )^{\frac{r-t}s-1} \left (z'-(b_1x_1+\cdots +b_tx_t  )\right )^{\frac{r'-t}s-1}\nonumber\\
 &\ll &\sum_{\substack{x_1,\dots,x_t\\ a_1x_1+\cdots +a_tx_t<z\\b_1x_1+\cdots +b_tx_t<z' } }(x_1\cdots x_t)^{-1+1/s}\left (z-(a_1x_1+\cdots +a_tx_t)  \right )^{-t/s} \left (z'-(b_1x_1+\cdots +b_tx_t  )\right )^{-t/s} \end{eqnarray*}
 Using the inequality  $AB\le A^2+B^2$, we get
 \begin{eqnarray*}
\sum_{\substack{\omega\sim \omega'}}^*P(E_{\omega} \cap E_{\omega'}) &\le &\sum_{\substack{x_1,\dots,x_t\\ a_1x_1+\cdots +a_tx_t<z } }(x_1\cdots x_t)^{-1+1/s}\left (z-(a_1x_1+\cdots +a_tx_t)  \right )^{-2t/s} \\
&+&\sum_{\substack{x_1,\dots,x_t\\ b_1x_1+\cdots +b_tx_t<z' } }(x_1\cdots x_t)^{-1+1/s} \left (z'-(b_1x_1+\cdots +b_tx_t  )\right )^{-2t/s}\\ (\text{Lemma }\ref{t1}, ii))\ &\ll & z^{-1/s}\log z.
 \end{eqnarray*}

 \bigskip

  \item $r=t$ and $r'\ge t+1$. In this case we have
\begin{eqnarray*}
\sum_{\substack{\omega\sim \omega'}}^*P(E_{\omega} \cap E_{\omega'}) &\le & \sum_{\substack{x_1,\dots,x_t\\ a_1x_1+\cdots +a_tx_t=z\\b_1x_1+\cdots +b_tx_t<z' } }(x_1\cdots x_t)^{-1+1/s}\\
&\times &\sum_{\substack{v_{t+1},\dots, v_{r'}\\b_{t+1}v_{t+1}+\dots +b_{r'}v_{r'}\\ =z'-(b_1x_1+\cdots +b_tx_t  )   }}(v_{t+1}\cdots v_{r'})^{-1+1/s}
\\
(\text{Lemma \ref{t1} i)})&\le  &\sum_{\substack{x_1,\dots,x_t\\ a_1x_1+\cdots +a_tx_t=z\\b_1x_1+\cdots +b_tx_t<z' } }(x_1\cdots x_t)^{-1+1/s}
 \times \left (z'-(b_1x_1+\cdots +b_tx_t  )\right )^{\frac{r'-t}s-1}\\
&\le &\sum_{\substack{x_1,\dots,x_t\\ a_1x_1+\cdots +a_tx_t=z } }(x_1\cdots x_t)^{-1+1/s}\ll z^{\frac ts-1}\ll z^{-1/s}.
\end{eqnarray*}
\item  $r'=t$ and $r\ge t+1$ is similar to the previous one.
\end{itemize}
\end{proof}
\begin{lemma}\label{dis}
Let $\alpha>0$ and the interval $I_i=[i,i+\alpha\log i]$. For any $i\le j$ we have
$$\sum_{\substack{\omega\sim \omega'\\ \omega\in \Omega_{I_i},\ \omega'\in \Omega_{I_j}}}P(E_{\omega}\cap E_{\omega'})\ll i^{-1/s}(\log i)^2(\log j).$$
\end{lemma}
\begin{proof}
\begin{eqnarray*}\sum_{\substack{\omega\sim \omega'\\ \omega\in \Omega_{I_i},\ \omega'\in \Omega_{I_j}}}P(E_{\omega}\cap E_{\omega'})&\le & \sum_{z\in I_i,\ z'\in I_j}\sum_{\substack{\omega\sim \omega'\\ \omega \in \Omega_z,\ \omega'\in \Omega_{z'}}}P(E_{\omega} \cap E_{\omega'})\\  &\ll&\sum_{z\in I_i,\ z'\in I_j}z^{-1/s}\log z\ll   (\log i)^2 (\log j) i^{-1/s}.\end{eqnarray*}
\end{proof}
\begin{lemma}\label{pp} We have
$$\prod_{\omega \in \Omega_{I_i}}P(E_{\omega}^c)=i^{-\alpha\lambda_s+o(1)}.  $$
\end{lemma}
\begin{proof}
We observe that
$$\prod_{z\in I_i}\prod_{\omega\in \Omega_z}P(E_{\omega}^c)\le    \prod_{\omega \in \Omega_{I_i}}P(E_{\omega}^c)\le  \prod_{\substack{\omega \in \Omega_{I_i}\\|\omega|=s}}P(E_{\omega}^c)=   \prod_{z\in I_i}\prod_{\substack{\omega\in \Omega_z\\ |\omega|=s}}P(E_{\omega}^c). $$

Writing $P(E_{\omega}^c)=1-P(E_{\omega})$ and taking logarithms we have
\begin{eqnarray*}\log \left (\prod_{z\in I_i}\prod_{\omega\in \Omega_z}P(E_{\omega}^c)\right )&=&\sum_{z\in I_i}\sum_{\omega\in \Omega_z}\log(1-P(E_{\omega}))\\
&\sim &-\sum_{z\in I_i}\sum_{\omega\in \Omega_z}P(E_{\omega})\\
(\text{Lemma }\ref{d})&\sim & -\sum_{z\in I_i}\lambda_s\\
&\sim & -\alpha\lambda_s\log i.\end{eqnarray*}
On the other hand,
\begin{eqnarray*}\log \left (\prod_{z\in I_i}\prod_{\substack{\omega\in \Omega_z\\|\omega|=s}}P(E_{\omega}^c)\right )&=&
\sum_{z\in I_i}\sum_{\substack{\omega\in \Omega_z\\|\omega|=s}}\log(1-P(E_{\omega}))\\
&\sim &-\sum_{z\in I_i}\sum_{\substack{\omega\in \Omega_z\\|\omega|=s}}P(E_{\omega})\\&= &-\sum_{z\in I_i}\sum_{\substack{x_1<\cdots <x_s\\x_1+\cdots +x_s=z}}\frac 1{s^s}(x_1\cdots x_s)^{-1+1/s}\\
(\text{Lemma \ref{t1} iii)})&\sim & -\lambda_s\alpha\log i.\end{eqnarray*}
\end{proof}
\begin{lemma}\label{igual}
 We have
$$P(F_i)=i^{-\alpha\lambda_s+o(1)}. $$
\end{lemma}
\begin{proof}
We observe that $$F_i=\bigcap_{\omega\in \Omega_{I_i}}E_{\omega}^c.$$
Since $P(E_{\omega})\le 1/2$ for any $\omega$, Theorem \ref{des1} applies and we have
$$\prod_{\omega\in \Omega_{I_i}}P(E_{\omega}^c)\le   P(F_i)\le \prod_{\omega\in \Omega_{I_i}}P(E_{\omega}^c)\times
\exp\left (2\sum_{\substack{\omega \sim \omega'\\ \omega,\omega' \in \Omega_{I_i}}}P(E_{\omega}\cap E_{\omega'})   \right). $$

After Lemma \ref{pp} we only need to prove $$\sum_{\substack{\omega \sim \omega'\\ \omega,\omega' \in \Omega_{I_i}}}P(E_{\omega}\cap E_{\omega'})=o(1).$$ But
it is a consequence of  Lemma \ref{dis} with $j=i$.
$$\sum_{\substack{\omega \sim \omega'\\ \omega,\ \omega' \in \Omega_{I_i}}}P(E_{\omega}\cap E_{\omega'})\ll i^{-1/s+o(1)}.$$
\end{proof}
\begin{lemma}\label{pp2} If $i<j$ and $I_i\cap I_j=\emptyset$ then
$$\prod_{\omega \in \Omega_{I_i}\cup \Omega_{I_j}}P(E_{\omega}^c)\le P(F_i)P(F_j)(1+O(j^{-1/s}\log j)).  $$
\end{lemma}
\begin{proof}
It is clear that
$$\prod_{\omega \in \Omega_{I_i}\cup \Omega_{I_j}}P(E_{\omega}^c)=\left (\prod_{\omega \in \Omega_{I_i}}P(E_{\omega}^c)\right )\left (\prod_{\omega \in  \Omega_{I_j}}P(E_{\omega}^c)\right )\left (\prod_{\omega \in \Omega_{I_i}\cap \Omega_{I_j}}P(E_{\omega}^c)\right )^{-1}.  $$
The lower bound of the Janson's inequality, applied to the first two products, gives
$$\prod_{\omega \in \Omega_{I_i}\cup \Omega_{I_j}}P(E_{\omega}^c)\le P(F_i)P(F_j)\left (\prod_{\omega \in \Omega_{I_i}\cap \Omega_{I_j}}P(E_{\omega}^c)\right )^{-1}.  $$

The logarithm of the last factor is
\begin{eqnarray*}-\sum_{\omega \in \Omega_{I_i}\cap \Omega_{I_j}}\log (1-P(E_{\omega}))&\sim  &\sum_{\omega \in \Omega_{I_i}\cap \Omega_{I_j}}P(E_{\omega})\end{eqnarray*}
Since $I_i\cap I_j=\emptyset$, if $\omega \in \Omega_{I_i}\cap \Omega_{I_j}$ then $|\omega|\le s-1$. Thus
\begin{eqnarray*} \sum_{\omega \in \Omega_{I_i}\cap \Omega_{I_j}}P(E_{\omega})&\le &\sum_{\substack{\omega \in  \Omega_{I_j}\\ |\omega|\le s-1}}P(E_{\omega})\\ &\le &\sum_{z\in I_j}\sum_{r\le s-1}\sum_{a_1+\cdots+a_r=s}\sum_{\substack{x_1,\dots,
 x_r\\a_1x_1+\cdots+a_rx_r=z}}(x_1\cdots x_r)^{-1+1/s}\\ (\text{Lemma }\ref{t1}\ i))&\ll &j^{-1/s}(\log j).  \end{eqnarray*}
Thus
$$\left (\prod_{\omega \in \Omega_{I_i}\cap \Omega_{I_j}}P(E_{\omega}^c)\right )^{-1}\le 1+O(j^{-1/s}(\log j))$$
which ends the proof of the Lemma.
\end{proof}
\subsection{End of the proof}
After these Lemmas we are ready to finish the proof of Theorem \ref{main}.

If $\alpha>1/\lambda_s$ then $$\sum_{i}P(F_i)=\sum_i i^{-\alpha\lambda_s+o(1)}< \infty$$ and Theorem \ref{Borel} implies that with probability $1$ only  finite many events $F_i$ occur. This proves that
$$\limsup_{k\to \infty}\frac{b_{k+1}-b_k}{\log b_k}\le 1/\lambda_s.$$

If $\alpha<1/\lambda_s$ then $$Z_n=\sum_{i\le n}P(F_i)=\sum_{i\le n}i^{-\alpha\lambda_s+o(1)}=n^{1-\alpha\lambda_s+o(1)}\to \infty.$$
If in addition \begin{equation}\label{l}\lim_{n\to \infty}\frac{\sum_{1\le i<j\le n} P(F_i\cap F_j)-P (F_i)P (F_j)}{Z_n^2}=0,\end{equation}
Theorem \ref{Borel} implies that with probability $1$ infinitely many events $F_i$ occur
and
$$\limsup_{k\to \infty}\frac{b_{k+1}-b_k}{\log b_k}\ge 1/\lambda_s.$$

We next prove \eqref{l}. We observe that $$ F_i\cap F_j=\bigcap_{\omega\in \Omega_{I_i}\cup \Omega_{I_j}}E_{\omega}^c, $$
so we can use Janson inequality to get \begin{eqnarray}\label{mu}P(F_i\cap F_j)&\le &\prod_{\omega\in \Omega_{I_i}\cup \Omega_{I_j}}P\big (E_{\omega}^c   \big )\times \exp \left  (2\sum_{\substack{\omega \sim \omega'\\ \omega,\omega'\in \Omega_{I_i}\cup \Omega_{I_j}}}P (E_{\omega} \cap E_{\omega'})\right  ).  \nonumber\end{eqnarray}
      Observe that
      \begin{eqnarray*}\sum_{\substack{\omega \sim \omega'\\ \omega,\omega'\in \Omega_{I_i}\cup \Omega_{I_j}}}P (E_{\omega} \cap E_{\omega'})&\le & \sum_{\substack{\omega \sim \omega'\\ \omega,\omega'\in \Omega_{I_i}}}P (E_{\omega} \cap E_{\omega'})\\ &+&  \sum_{\substack{\omega \sim \omega'\\ \omega,\omega'\in \Omega_{I_j}}}P (E_{\omega} \cap E_{\omega'})\\ &+& \sum_{\substack{\omega \sim \omega'\\ \omega\in \Omega_{I_i},\ \omega'\in \Omega_{I_j}}}P (E_{\omega} \cap E_{\omega'}).\end{eqnarray*}

Applying Lemma \ref{dis} to the three sums we have   \begin{eqnarray*}\sum_{\substack{\omega \sim \omega'\\ \omega,\omega'\in \Omega_{I_i}\cup \Omega_{I_j}}}P (E_{\omega} \cap E_{\omega'})\ll i^{-1/s}(\log i)^3+ j^{-1/s}(\log j)^3+i^{-1/s}(\log i)^2(\log j),\end{eqnarray*}
and so
 \begin{equation}\label{exp}\exp \left (2\sum_{\substack{\omega \sim \omega'\\ \omega,\omega'\in \Omega_{I_i}\cup \Omega_{I_j}}}P (E_{\omega} \cap E_{\omega'})\right)\le 1+O\left (i^{-1/s}(\log i)^2(\log j)\right ).\end{equation}

 Thus,
 \begin{equation}\label{ine}P(F_i\cap F_j)\le \prod_{\omega\in \Omega_{I_i}\cup \Omega_{I_j}}P\big (E_{\omega}^c   \big )\times \big (1+O(i^{-1/s}(\log i)^2(\log j))\big ). \end{equation}

 Since $\alpha<\lambda_s$, the number $\beta=(1-\alpha\lambda_s)/2$ is positive. Now we split the sum in \eqref{l} into three sums:

\begin{eqnarray*}\Delta_{1n}&=&\sum_{\substack{1\le i<j\le n\\ n^{\beta}<i<j-\alpha\log j}} P(F_i\cap F_j)-P (F_i)P (F_j)\\
\Delta_{2n}&=&\sum_{\substack{1\le i<j\le n\\ i\le n^{\beta}}} P(F_i\cap F_j)-P (F_i)P (F_j)\\
\Delta_{3n}&=&\sum_{\substack{1\le i<j\le n\\ j-\log j\le i\le j}} P(F_i\cap F_j)-P (F_i)P (F_j)
\end{eqnarray*}

 \begin{itemize}
  \item [i)] Estimate of $\Delta_{1n}$. Since in this case we have $I_i\cap I_j=\emptyset $, we can apply Lemma \ref{pp2} to \eqref{ine} to get
  $$\prod_{\omega\in \Omega_{I_i}\cup \Omega_{I_j}}P\big (E_{\omega}^c   \big )\le P(F_i)P(F_j)(1+O(j^{-1/s}\log j)).$$
  This inequality and \eqref{ine} gives
  $$P(F_i\cap F_j)\le P(F_i)P(F_j)\times \big (1+O(i^{-1/s}(\log i)^2(\log j))\big ),$$
so
\begin{eqnarray*}P(F_i\cap F_j)-P(F_i)P(F_j) &\ll &P(F_i)P(F_j)i^{-1/s}(\log i)^2(\log j)\\
&\ll & n^{-\beta/s+o(1)}P(F_i)P(F_j).
\end{eqnarray*}
Thus
\begin{eqnarray}\label{f1}\Delta_{1n}&\ll & n^{-\beta/s+o(1)}\sum_{i,j\le n}P(F_i)P(F_j)\ll n^{-\beta/s+o(1)}Z_n^2.\end{eqnarray}

  \item [ii)]Estimate of $\Delta_{2n}$.  In this case we use the crude estimate \begin{equation}\label{crude}P(F_i\cap F_j)-P(F_i)P(F_j)\le P(F_i\cap F_j)\le P(F_j).\end{equation} We have
\begin{eqnarray}\label{f2} \Delta_{2n} &\le &\sum_{j\le n}\sum_{i\le j^{\beta}}P(F_j)\le \sum_{j\le n}j^{\beta}P(F_j)\le n^{\beta}Z_n\le n^{-\beta+o(1)}Z_n^2, \end{eqnarray}
since $Z_n=n^{1-\alpha\lambda_s+o(1)}=n^{2\beta+o(1)}$.
\item[iii)]Estimate of $\Delta_{3n}$. Again we use \eqref{crude} and we have
\begin{eqnarray}\label{f3} \Delta_{3n} &\le &\sum_{j\le n}\sum_{j-\alpha\log j\le i\le j}P(F_j)\le \alpha\log n\sum_{j\le n}P(F_j)\le n^{-2\beta+o(1)}Z^2_n. \end{eqnarray}
\end{itemize}

Finally, using the estimates in \eqref{f1},\eqref{f2} and\eqref{f3} we have
$$\frac{\sum_{1\le i<j\le n} P(F_i\cap F_j)-P (F_i)P (F_j)}{Z_n^2}\ll n^{-\beta/s+o(1)}+  n^{-\beta+o(1)}+n^{-2\beta+o(1)}\to 0.$$
This ends the proof of \eqref{l} and hence that of Theorem \ref{main}.

\section*{Acknowledgment}The first author was supported by grants MTM 2011-22851 of MICINN  and ICMAT Severo Ochoa project SEV-2011-0087.  Both authors are thankful to \'Ecole Polytechnique which made their collaboration easier.

\end{document}